\pdfoutput=1
\RequirePackage{ifpdf}
\ifpdf % We are running pdfTeX in pdf mode
\documentclass[pdftex]{sigma}
\else
\documentclass{sigma}
\fi

\usepackage{subdepth}

\numberwithin{equation}{section}

\newtheorem{Theorem}{Theorem}[section]
\newtheorem*{Theorem*}{Theorem}

\newtheorem{Lemma}[Theorem]{Lemma}
\newtheorem{Proposition}[Theorem]{Proposition}
 { \theoremstyle{definition}

\newtheorem{Remark}[Theorem]{Remark} }

\def\E {\mathbb{E}}

\def\R {\mathbb{R}}

\def\dd {\hskip1pt\mathrm{d}}
\def\tr {\mathop{\mathrm{tr}}}

\def\diag{\mathop{\mathrm{diag}}}
\def\Or{\mathop{\mathrm{Or}}}

\newcommand{\eqdist}{\overset{\mathcal{L}}{=}}

\newcommand{\sym}{{{\hskip-1pt}\mathop{\mathrm{sym}}}}

\begin{document}

\allowdisplaybreaks

\renewcommand{\thefootnote}{}

\newcommand{\arXivNumber}{2402.08663}

\renewcommand{\PaperNumber}{094}

\FirstPageHeading

\ShortArticleName{Complete Asymptotic Expansions for the Normalizing Constants}

\ArticleName{Complete Asymptotic Expansions for the Normalizing\\ Constants of High-Dimensional Matrix Bingham\\ and Matrix Langevin Distributions\footnote{This paper is a~contribution to the Special Issue on Asymptotics and Applications of Special Functions in Memory of Richard Paris. The~full collection is available at \href{https://www.emis.de/journals/SIGMA/Paris.html}{https://www.emis.de/journals/SIGMA/Paris.html}}}

\Author{Armine BAGYAN and Donald RICHARDS}
\AuthorNameForHeading{A.~Bagyan and D.~Richards}

\Address{Department of Statistics, Pennsylvania State University, University Park, PA 16802, USA}
\Email{\href{mailto:aub171@psu.edu}{aub171@psu.edu}, \href{mailto:dsr11@psu.edu}{dsr11@psu.edu}}

\ArticleDates{Received February 28, 2024, in final form October 10, 2024; Published online October 21, 2024}

\Abstract{For positive integers $d$ and $p$ such that $d \ge p$, let $\R^{d \times p}$ denote the set of $d \times p$ real matrices, $I_p$ be the identity matrix of order $p$, and \smash{$V_{d,p} = \bigl\{x \in \R^{d \times p} \mid x'x = I_p\bigr\}$} be the Stiefel manifold in \smash{$\R^{d \times p}$}. Complete asymptotic expansions as $d \to \infty$ are obtained for the normalizing constants of the matrix Bingham and matrix Langevin probability distributions on $V_{d,p}$. The accuracy of each truncated expansion is strictly increasing in $d$; also, for sufficiently large $d$, the accuracy is strictly increasing in $m$, the number of terms in the truncated expansion. Lower bounds are obtained for the truncated expansions when the matrix parameters of the matrix Bingham distribution are positive definite and when the matrix parameter of the matrix Langevin distribution is of full rank. These results are applied to obtain the rates of convergence of the asymptotic expansions as both $d \to \infty$ and $p \to \infty$. Values of $d$ and $p$ arising in numerous data sets are used to illustrate the rate of convergence of the truncated approximations as $d$ or $m$ increases. These results extend recently-obtained asymptotic expansions for the normalizing constants of the high-dimensional Bingham distributions.}


\Classification{60E05; 62H11; 62E20; 62R30}

\renewcommand{\thefootnote}{\arabic{footnote}}
\setcounter{footnote}{0}

\section{Introduction}
\label{sec_introduction}

For positive integers $d$ and $p$ such that $d \ge p$ let $\R^{d \times p}$ denote the set of $d \times p$ real matrices, $I_p$ be the identity matrix of order $p$, and \smash{$V_{d,p} = \bigl\{x \in \R^{d \times p} \mid x'x = I_p\bigr\}$} denote the \textit{Stiefel manifold}. The manifold $V_{d,p}$ is endowed with a unique probability measure, denoted by $\dd x$, that is invariant under the change of variables $x \mapsto HxK$, where $x \in V_{d,p}$, and $H$ and $K$ are orthogonal matrices of orders $d \times d$ and $p \times p$, respectively (see Chikuse \cite[p.~16]{Chikuse}, Muirhead \cite[p.~70]{Muirhead}).

In this article, we first consider the matrix Bingham probability distributions on $V_{d,p}$. For non-zero symmetric matrices \smash{$A \in \R^{p \times p}$} and \smash{$\Sigma \in \R^{d \times d}$}, a random matrix $X \in V_{d,p}$ is said to have a {\it $($full$)$ matrix Bingham distribution with parameters $A$ and $\Sigma$} if the probability density function of $X$, with respect to the invariant measure $\dd x$, is
\begin{equation}
\label{eq_mBingham_pdf}
\phi(x;A,\Sigma) = [\Phi_{d,p}(A,\Sigma)]^{-1} \exp\bigl(\tr Ax'\Sigma x\bigr), \qquad x \in V_{d,p},
\end{equation}
where $\Phi_{d,p}(A,\Sigma)$ is the normalizing constant. The matrix Bingham family is now ubiquitous in the statistical analysis of directional data and shape analysis, and the extant literature now provides many properties of the family and its applications in fields including astronomy, biology, computer vision, medicine, meteorology, physics and psychology; see, e.g., \cite{BinghamCR,Chikuse,Dryden,JuppMardia2,KhatriMardia,KumeWood,KumePrestonWood,MardiaJupp,Prentice}.

We also consider the matrix Langevin (or matrix von Mises--Fisher) family of probability distributions on $V_{d,p}$. For non-zero \smash{$B \in \R^{d \times p}$}, a random matrix $X \in V_{d,p}$ is said to have a~{\it matrix Langevin distribution with parameter} $B$ if the probability density function of $X$, with respect to the invariant measure $\dd x$, is
\begin{equation}
\label{eq_mLangevin_pdf}
\psi(x;B) = [\Psi_{d,p}(B)]^{-1} \exp(\tr B'x), \qquad x \in V_{d,p},
\end{equation}
where $\Psi_{d,p}(B)$ is the normalizing constant. These distributions too are widely used in the statistical analysis of shape and directional data; see, e.g., Chikuse~\cite{Chikuse}, Mardia and Jupp~\cite{MardiaJupp}, Oualkacha and Rivest~\cite{OualkachaRivest}.

A major difficulty in statistical inference with the matrix Bingham and matrix Langevin distributions is the calculation of the constants $\Phi_{d,p}(A,\Sigma)$ and $\Psi_{d,p}(B)$. Some authors have described this issue as the ``main obstacle'' to maximum likelihood estimation because the normalizing constants, although known analytically, are difficult to compute in practice (Mantoux, et al.\ \cite[pp.~4,~11]{Mantoux_etal}). For that reason, substantial attention has been paid to the problem of calculating those constants (see Kume and Wood \cite{KumeWood}, and Kume, Preston and Wood \cite{KumePrestonWood}).

In a recent article \cite{BagyanRichards}, we studied the Bingham distributions for the case in which $p = 1$. Under a growth condition on $\Sigma$ we obtained for $\Phi_{d,1}(1,\Sigma)$ a complete asymptotic series expansion in negative powers of $d$, and we derived a bound for the tail of that series upon truncation after, say, $m$ terms. In \cite{BagyanRichards}, we used an explicit formula for the zonal polynomials indexed by partitions with one part; however, as no explicit formula is known for the zonal polynomials in general, a more powerful approach is needed for general $p$, and the purpose of the present article is to provide such an approach.

In the present article, we obtain for all $d$ and $p$ complete asymptotic expansions of $\Phi_{d,p}(A,\Sigma)$ and $\Psi_{d,p}(B)$ in negative powers of $d$; these expansions are derived subject to growth conditions on the Frobenius norms of $\Sigma$ and $B$ as $d \to \infty$. We accomplish these results using detailed properties of the zonal polynomials, including an integral representation due to James \cite{James73}.

We also derive an infinite series bound and a closed-form bound for the remainder term in each asymptotic expansion. These bounds, as they are valid for all $d \ge p$, extend the results in~\cite{BagyanRichards}. We show further that the new bounds decrease strictly to zero as $d$, or as $m$, the number of terms in the truncated expansion, increases; equivalently, the accuracy of each truncated asymptotic expansion is strictly increasing in $d$ or $m$. We devote close attention to the closed-form upper bound as it is likely to be more widely used than its infinite series counterpart; in particular, we derive conditions on the values of the parameters such that the closed-form bound is unimodal in $m$.

When the matrices $A$ and $\Sigma$ in \eqref{eq_mBingham_pdf} are positive definite, we obtain a lower bound on the truncated expansion for $\Phi_{d,p}(A,\Sigma)$. Further, we derive a lower bound for the truncated expansion of $\Psi_{d,p}(B)$ when $B$ is of full rank. These lower bounds can be expressed in closed-form, and they decrease to zero as $d \to \infty$ or $m \to \infty$.

The contents of the article are as follows. In Section \ref{sec_matrix_bingham}, we obtain the asymptotic expansion for $\Phi_{d,p}(A,\Sigma)$ and we remark on the properties of the expansion, and in Section \ref{sec_matrix_langevin}, we obtain a similar expansion for $\Psi_{d,p}(B)$. In Section \ref{sec_implications}, we survey the range of values of $d$ and $p$ that have appeared in the literature on analyses of data on Stiefel manifolds, describe the rate of convergence to zero of each bound as $d$ or $m$ increases, and conclude by noting that our results extend to more general classes of distributions on the Stiefel and Grassmann manifolds and to the complex analogs of the matrix Bingham and matrix Langevin distributions. We provide in Section \ref{sec_zonal_polys} some necessary properties of the zonal polynomials, and all proofs are given in Sections \ref{sec_ineqs}--\ref{sec_proofs_Psi_bounds}.

\section[Complete asymptotics for the normalizing constant of the high-dimensional matrix Bingham distribution]{Complete asymptotics for the normalizing constant\\ of the high-dimensional matrix Bingham distribution}
\label{sec_matrix_bingham}

A \textit{partition} $\kappa = (\kappa_1,\ldots,\kappa_d)$ is a vector of integers such that $\kappa_1 \ge \cdots \ge \kappa_d \ge 0$. The non-zero~$\kappa_j$~are called the \textit{parts} of $\kappa$. The \textit{length} of $\kappa$, denoted by $\ell(\kappa)$, is the number of non-zero~$\kappa_j$,~$j=1,\ldots,d$, and the \textit{weight} of $\kappa$ is $|\kappa| = \kappa_1+\cdots+\kappa_d$.

Define the \textit{shifted factorial}, $(a)_k = a(a+1)(a+2) \cdots (a+k-1)$, where $a \in \mathbb{R}$ and ${k=0,1,2,\ldots}$; in particular, $(a)_0 \equiv 1$ since the defining product is empty. For $a \in \mathbb{R}$ and a partition ${\kappa = (\kappa_1,\ldots,\kappa_d)}$, the {\it partitional shifted factorial} is %defined as
\begin{equation}
\label{eq_partit_fact}
(a)_\kappa = \prod_{i=1}^d \left(a-\frac{1}{2}(i-1)\right)_{\kappa_i}.
\end{equation}

Assume that \smash{$\R^{d \times d}_\sym$} denote the space of real symmetric $d \times d$ matrices. For a partition $\kappa = (\kappa_1,\ldots,\kappa_d)$ and \smash{$\Sigma \in \R^{d \times d}_\sym$}, the \textit{zonal polynomial} $C_\kappa(\Sigma)$ is an eigenfunction of the Laplace--Beltrami operator in the eigenvalues of $\Sigma$ (see \cite{James68}, \cite[Chapter 7]{Muirhead} and \cite{Richards10}). Further, $C_\kappa(\Sigma)$~depends only on the eigenvalues of $\Sigma$, and is symmetric in those eigenvalues.

Tables of the polynomials $C_\kappa(\Sigma)$ for $|\kappa| \le 12$ were provided by James \cite[pp.~498--499]{James64} and Parkhurst and James \cite{ParkhurstJames}. Jiu and Koutschan \cite{JiuKoutschan} have provided software to calculate rapidly the zonal polynomials to high values of $|\kappa|$, so the asymptotic expansions obtained in this article can be calculated for practical values of $m$.

Let $X \in V_{d,p}$ be a random matrix having the matrix Bingham distribution with the density function $\phi(x;A,\Sigma)$ as in \eqref{eq_mBingham_pdf}. It is well known that the normalizing constant $\Phi_{d,p}(A,\Sigma)$ can be expressed as a generalized hypergeometric function of two matrix arguments, denoted by~${}_0F_0(A,\Sigma)$, or by the zonal polynomial expansion of that function,
\begin{equation}
\label{eq_normaliz_const_B}
\Phi_{d,p}(A,\Sigma) = {}_0F_0(A,\Sigma) = \sum_{k=0}^\infty \frac{1}{k!} \sum_{|\kappa|=k} \frac{C_\kappa(A) C_\kappa(\Sigma)}{C_\kappa(I_d)},
\end{equation}
where the inner sum is over all partitions $\kappa$ such that $\ell(\kappa) \le p$. This formula for $\Phi_{d,p}(A,\Sigma)$ in terms of ${}_0F_0(A,\Sigma)$ dates to De Waal \cite{DeWaal}; cf.\ Bingham \cite{Bingham}, Chikuse \cite[p.~108]{Chikuse}. By applying the methods of \cite[Theorem~6.3]{GrossRichards}, we find that the series \eqref{eq_normaliz_const_B} converges absolutely for all $(A,\Sigma)$ and uniformly on compact regions.

For the special case in which $A = I_p$, on applying to \eqref{eq_normaliz_const_B} the identity \eqref{eq_kappa_invar}, written in the form $C_\kappa(I_p)/C_\kappa(I_d) = (p/2)_\kappa/(d/2)_\kappa$, we obtain the normalizing constant in terms of a confluent hypergeometric function of matrix argument,
%\label{eq_normaliz_const_hgf}
\[
\Phi_{d,p}(I_p,\Sigma) = \sum_{k=0}^\infty \frac{1}{k!} \sum_{|\kappa|=k} \frac{(p/2)_\kappa}{(d/2)_\kappa} C_\kappa(\Sigma) = {}_1F_1\left(\frac12 p;\frac12 d;\Sigma\right),
\]
as shown by various authors, e.g., Bingham \cite{Bingham}, Chikuse \cite[p.~33]{Chikuse} and Muirhead \cite[p.~288]{Muirhead}.

Similar to \cite[equation~(2.4)]{BagyanRichards}, when $d$ is large, our asymptotic approximation to $\Phi_{d,p}(A,\Sigma)$ is the sum of the terms up to degree $m-1$ of the series \eqref{eq_normaliz_const_B}, i.e.,
\begin{equation}
\label{eq_trunc_expan}
\Phi_{d,p}(A,\Sigma) \approx \sum_{k=0}^{m-1} \frac{1}{k!} \sum_{|\kappa|=k} \frac{C_\kappa(A) C_\kappa(\Sigma)}{C_\kappa(I_d)},
\end{equation}
$m \ge 2$, and now we study the remainder series
\begin{equation}
\label{eq_remainder_term}
\Phi_{d,p;m}(A,\Sigma) = \sum_{k=m}^\infty \frac{1}{k!} \sum_{|\kappa|=k} \frac{C_\kappa(A) C_\kappa(\Sigma)}{C_\kappa(I_d)}.
\end{equation}

In the sequel, we will encounter the constants \smash{$\alpha_p = (2\pi)^{-(p-1)/4p} p^{1/4p}$} and $\gamma_1 = \bigl(\sqrt{3}+1\bigr)/2$, arising in Lemmas \ref{lemma_factorial_ineq} and \ref{lemma_ineqs}, respectively. Also, if $a_1,\ldots,a_p$ are the eigenvalues of $A$, then we~define $A_+ = \diag(|a_1|,\ldots,|a_p|)$, the diagonal matrix with diagonal entries $|a_1|,\ldots,|a_p|$. In~the following result, the proof of which is given in Section \ref{sec_proofs_Phi_bounds}, we extend a result in~\cite{BagyanRichards} by~deriving an upper bound for $\Phi_{d,p;m}(A,\Sigma)$ in terms of the Frobenius norm, \smash{$\|\Sigma\| = \bigl[\tr\bigl(\Sigma^2\bigr)\bigr]^{1/2}$}.

\begin{Theorem}
\label{thm_Phi_m_bound_Bingham}
Let $A \in \R^{p \times p}_\sym$, \smash{$\Sigma \in \R^{d \times d}_\sym$}, and suppose that \smash{$\|\Sigma\| \le \gamma_0 d^{r/2}$}, where $\gamma_0 > 0$ and $0 \le r < 1$. Then for all $m \ge 2$,
\begin{align}
|\Phi_{d,p;m}(A,\Sigma)| \le{}& \alpha_p \sum_{k=m}^\infty \frac{\bigl[\gamma_0 \gamma_1 p^{1/2} (\tr A_+) d^{-(1-r)/2}\bigr]^k}{(k!)^{1/2}} \label{eq_Phi_m_bound} \\
\le{}& \alpha_p (4 {\rm e}/\pi)^{1/4} ({\rm e}/m)^{(m/2)-(1/4)}
\bigl[\gamma_0 \gamma_1 p^{1/2} (\tr A_+) d^{-(1-r)/2}\bigr]^m \nonumber \\
& \times \exp\bigl(c_m \bigl[\gamma_0 \gamma_1 p^{1/2} (\tr A_+) d^{-(1-r)/2}\bigr]^2/2\bigr), \label{eq_Phi_m_boundw}
\end{align}
where $c_m = \bigl(1+m^{-1}\bigr)^{-m} {\rm e}$. In particular, $\Phi_{d,p;m}(A,\Sigma) = O\bigl(d^{-(1-r)m/2}\bigr)$ as $d \to \infty$.
\end{Theorem}

\begin{Remark}\label{rem_bound_Bingham}\quad
\begin{itemize}\itemsep=0pt
\item[(i)] The growth condition, \smash{$\|\Sigma\| \le \gamma_0 d^{r/2}$} with $\gamma_0 > 0$ and $r \in [0,1)$, is the same as was assumed in \cite{BagyanRichards} for the Bingham distribution, i.e., the case $p=1$. Further, the bounds \eqref{eq_Phi_m_bound} and~\eqref{eq_Phi_m_boundw} provide a rate of convergence, viz., \smash{$\Phi_{d,p;m}(A,\Sigma) = O\bigl(d^{-m(1-r)/2}\bigr)$} as $d \to \infty$, that is the same as we obtained in \cite{BagyanRichards} for the case in which $p = 1$. On the other hand, unlike the bounds in \cite{BagyanRichards}, which are valid only for sufficiently large $d$, the bounds \eqref{eq_Phi_m_bound} and \eqref{eq_Phi_m_boundw} are valid for all $d \ge p \ge 1$ and even for all $r \in \R$.

\item[(ii)] The bounds \eqref{eq_Phi_m_bound} and \eqref{eq_Phi_m_boundw} each are increasing in $\gamma_0$ and in $r$; this is consistent with the restriction \smash{$\|\Sigma\| \le \gamma_0 d^{r/2}$} representing a larger matrix region as $\gamma_0$ or $r$ increases.

\item[(iii)] Each bound is also strictly increasing in $p$ or $\tr(A_+)$ for fixed $d$, $m$, $\gamma_0$, and $r$. This follows from the fact that, in the $k$th summand in the bound in \eqref{eq_Phi_m_bound}, the term \smash{$\alpha_p p^{k/2}$} is strictly increasing in $p$; therefore, the bound itself is strictly increasing. A similar argument also yields the same conclusion for the bound in \eqref{eq_Phi_m_boundw}.

\item[(iv)] It is evident that the bound in \eqref{eq_Phi_m_bound}, as a function of $d$ or $m$, decreases strictly to zero as~${d \to \infty}$ or $m \to \infty$. Equivalently, the accuracy of the truncated asymptotic expansion~\eqref{eq_trunc_expan} is strictly increasing in $d$ or $m$.
Further, it is evident that the closed-form bound in~\eqref{eq_Phi_m_boundw} is strictly decreasing in $d$.

\item[(v)] The bound \eqref{eq_Phi_m_bound} on $\Phi_{d,p;m}(A,\Sigma)$ also allows us to obtain asymptotics as both $d$ and $p$ tend to infinity. By \eqref{eq_alpha_p}, \smash{$\alpha_p \to (2\pi)^{-1/4}$} as $p \to \infty$, and therefore we deduce that $\Phi_{d,p;m}(A,\Sigma) = O(1)$ if \smash{$(\tr A_+)^2 p d^{-(1-r)} = O(1)$}, i.e., if \smash{$(\tr A_+)^2 p = O\bigl(d^{1-r}\bigr)$}.
 \end{itemize}
\end{Remark}

\begin{Remark}
\label{rem_boundw_Bingham}\quad
\begin{itemize}\itemsep=0pt
\item[(i)] The behavior of the closed-form bound \eqref{eq_Phi_m_boundw}, as a function of $m$, is more intricate. We~shall prove in Section \ref{sec_proofs_Phi_bounds} that the bound \eqref{eq_Phi_m_boundw} is strictly decreasing in $m$ if $d$ is sufficiently large, precisely, if \smash{$d \ge \bigl[\gamma_0 \gamma_1 p^{1/2} (\tr A_+)\bigr]^{2/(1-r)}$}.

\item[(ii)] For \smash{$d < \bigl[\gamma_0 \gamma_1 p^{1/2} (\tr A_+)\bigr]^{2/(1-r)}$}, the bound \eqref{eq_Phi_m_boundw}, as a function of $m$, is unimodal. Therefore, for smaller values of $d$, the bound \eqref{eq_Phi_m_boundw} may increase for small values of $m$, but it eventually attains a maximum value and then decreases strictly to zero thereafter. This shows again that if $d$ is small, then higher values of $m$ may need to be chosen in order for the approximation \eqref{eq_trunc_expan} to result in small errors, and the bounds \eqref{eq_Phi_m_bound} and \eqref{eq_Phi_m_boundw} enable us to ascertain suitable values of $m$ in order for the approximation \eqref{eq_trunc_expan} to achieve any required level of accuracy.
\end{itemize}
\end{Remark}

To close this section, we obtain a lower bound for the remainder term $\Phi_{d,p;m}(A,\Sigma)$ for the case in which $A$ and $\Sigma$ are arbitrary positive definite matrices.

\begin{Theorem}
\label{thm_Phi_m_lower_bound_Bingham}
Let $A \in \R^{p \times p}_\sym$ and \smash{$\Sigma \in \R^{d \times d}_\sym$} be positive definite, let \smash{$\sigma_{(1)} \ge \cdots \ge \sigma_{(d)}$} denote the ordered eigenvalues of $\Sigma$, and set $N_d = \tfrac12(d+2)(d-1)$. Then for all $m \ge 1$,
\begin{equation}
\label{eq_Phi_m_lower_bound_Bingham}
\Phi_{d,p;m}(A,\Sigma) \ge (2+m)^{-mN_d} (1+m)^{-(1+m)N_d} \sum_{k=m}^\infty \left(\frac{1+m}{2+m}\right)^{kN_d} \frac{(\sigma_{(p)} \tr A)^k}{k!}.
\end{equation}
\end{Theorem}

\begin{Remark}
\label{rem_Phi_m_lower_bound_Bingham}
Let $W$ be a discrete random variable having a Poisson distribution with parameter~\smash{$\mu = \bigl(\frac{1+m}{2+m}\bigr)^{N_d} \sigma_{(p)} \tr A$}.
Then \eqref{eq_Phi_m_lower_bound_Bingham} can be written as
\begin{equation}
\label{eq_S_tau_lb}
\Phi_{d,p;m}(A,\Sigma) \ge (2+m)^{mN_d} (1+m)^{-(1+m)N_d} {\rm e}^{\mu} \mathbb{P}(W \ge m),
\end{equation}
and this result reveals several properties of the lower bound. First, by using the formula
${\mathbb{P}(W \ge m) = 1 - \mathbb{P}(W < m)}$,
it is seen that the right-hand side of \eqref{eq_S_tau_lb} can be expressed in closed-form. Second, as $m \to \infty$, it is simple to show that the right-hand side of \eqref{eq_S_tau_lb} is asymptotically equal to ${\rm e}^{N_d+\mu} m^{-N_d} \mathbb{P}(W \ge m)$, where $W$ is Poisson-distributed with parameter~${\mu \simeq \sigma_{(p)}\tr A}$, and this asymptotic lower bound converges to $0$ as $m \to \infty$.

Third, suppose that $\mu$ is large. By the central limit theorem, $(W-\mu)/\sqrt{\mu}$ is approximately distributed as $Z$, a random variable having the standard normal distribution (with mean $0$ and standard deviation $1$). Applying the usual correction for continuity that is used whenever, the distribution of a discrete random variable is approximated by the distribution of a continuous random variable, we obtain
\[
\mathbb{P}(W \ge m) = \mathbb{P}\left(W \ge m-\frac12\right) = \mathbb{P}\left(\frac{W-\mu}{\sqrt{\mu}} \ge \frac{m-\tfrac12-\mu}{\sqrt{\mu}}\right) \simeq \mathbb{P}\left(Z \ge \frac{m-\tfrac12-\mu}{\sqrt{\mu}}\right),
\]
which can be evaluated numerically using the standard normal distribution function.

Fourth, it is evident that the lower bound in \eqref{eq_Phi_m_lower_bound_Bingham} decreases strictly to $0$ as $d \to \infty$.

\end{Remark}

\section[Complete asymptotics for the normalizing constant of the high-dimensional matrix Langevin distribution]{Complete asymptotics for the normalizing constant\\ of the high-dimensional matrix Langevin distribution}
\label{sec_matrix_langevin}

Let $X \in V_{d,p}$ be a random matrix having the matrix Langevin distribution with the density function $\psi(x;A,\Sigma)$ given in \eqref{eq_mLangevin_pdf}. The normalizing constant $\Psi_{d,p}(B)$ can be expressed as a~generalized hypergeometric function of matrix argument,
\[
\Psi_{d,p}(B) = {}_0F_1\left(\frac12;\frac14B'B\right) = \sum_{k=0}^\infty \frac{1}{k!} \sum_{|\kappa|=k} \frac{C_\kappa\bigl(\frac14 B'B\bigr)}{(d/2)_\kappa},
\]
cf.\ Chikuse \cite[p.~31]{Chikuse}. By applying \cite[Theorem 6.3]{GrossRichards}, we find that this series converges absolutely for all \smash{$B \in \R^{d \times p}$} and uniformly on compact regions.

Similar to the approach in Section \ref{sec_matrix_bingham}, we use as an asymptotic approximation to $\Psi_{d,p}(B)$ the truncated series,
\[
\Psi_{d,p}(B) \approx \sum_{k=0}^{m-1} \frac{1}{k!} \sum_{|\kappa|=k} \frac{C_\kappa\bigl(\frac14 B'B\bigr)}{(d/2)_\kappa},
\]
$m \ge 2$, and we derive for the remainder series,
\begin{equation}
\label{eq_Psi_m_Langevin}
\Psi_{d,p;m}(B) = \sum_{k=m}^\infty \frac{1}{k!} \sum_{|\kappa|=k} \frac{C_\kappa\bigl(\frac14 B'B\bigr)}{(d/2)_\kappa},
\end{equation}
an upper bound that is a complete asymptotic expansion in negative powers of $d$. The proof of the following result is given in Section \ref{sec_proofs_Psi_bounds}, and the constants \smash{$\alpha_p = (2\pi)^{-(p-1)/4p} p^{1/4p}$} and $\gamma_1 = \bigl(\sqrt{3}+1\bigr)/2$ again appear in the proof.

\begin{Theorem}
\label{thm_Psi_m_bound_Langevin}
Let \smash{$B \in \R^{d \times p}$} be such that \smash{$\|B\| \le 2 \gamma_0^{1/2} d^{r/4}$}, where $\gamma_0 > 0$ and $0 \le r < 3$. Then for all $m \ge 2$, $\Psi_{d,p;m}(B) \ge 0$ and
\begin{align}
\Psi_{d,p;m}(B) \le{}& \alpha_p \sum_{k=m}^\infty \frac{\bigl[2 \gamma_0 \gamma_1 p^{5/2} d^{-(3-r)/2}\bigr]^k}{(k!)^{1/2}} \label{eq_Psi_m_bound} \\
\le{}& \alpha_p (4{\rm e}/\pi)^{1/4} ({\rm e}/m)^{(m/2)-(1/4)} \bigl[2 \gamma_0 \gamma_1 p^{5/2} d^{-(3-r)/2}\bigr]^m \nonumber \\
& \times \exp\bigl(c_m \bigl[2 \gamma_0 \gamma_1 p^{5/2} d^{-(3-r)/2}\bigr]^2/2\bigr). \label{eq_Psi_m_boundw}
\end{align}
In particular, $\Psi_{d,p;m}(B) = O\bigl(d^{-(3-r)m/2}\bigr)$ as $d \to \infty$.
\end{Theorem}

We note that each comment in Remark \ref{rem_bound_Bingham} has a direct analog for the matrix Langevin distribution. In particular, the bound in \eqref{eq_Psi_m_bound} is strictly decreasing in $d$ and $m$; strictly increasing in $\gamma_0$ and $r$; and strictly increasing in $p$ for fixed values of all other parameters. Also, the bound in \eqref{eq_Psi_m_boundw} is strictly decreasing in $d$, and it is strictly decreasing in $m$ if $d$ is sufficiently large.

As regards asymptotics for $d, p \to \infty$, by proceeding as in Remark \ref{rem_bound_Bingham}, we deduce that $\Psi_{d,p;m}(B) = O(1)$ if $p^{5/2} d^{-(3-r)/2} = O(1)$, i.e., if $p = O\bigl(d^{(3-r)/5}\bigr)$.

In closing this section, we provide a lower bound for $\Psi_{d,p;m}(B)$.

\begin{Theorem}
\label{thm_Psi_m_lower_bound_Langevin}
Let \smash{$B \in \R^{d \times p}$} such that $B$ is of rank $p$. Let \smash{$\beta_{(1)} \ge \cdots \ge \beta_{(p)}$} denote the ordered eigenvalues of $\tfrac14 B'B$, $N_p = \tfrac12(p+2)(p-1)$, and
\smash{$\mu = \bigl(\frac{1+m}{2+m}\bigr)^{N_p} 2 p \beta_{(p)} {\rm e}^{-1}$}.
Then for all $m \ge 1$,
\begin{equation}
\label{eq_Psi_m_lower_bound_Langevin}
\Psi_{d,p;m}(B) \ge \frac{(1+m)^{-N_p} (d+m)^{-m} (2p\beta_{(p)})^m}{m!} {}_1F_2(1;m+1,d+1+m;\mu).
\end{equation}
\end{Theorem}

In proving this result, we will also deduce the weaker bound,
\[
\Psi_{d,p;m}(B) \ge \frac{(1+m)^{-N_p} (d+m)^{-m} (2p\beta_{(p)})^m}{m!},
\]
a bound that is still noteworthy as it is a single-term lower bound, it is in closed-form, and it converges to $0$ as $d \to \infty$ or $m \to \infty$.

\section{Implications for applications}
\label{sec_implications}

The bounds derived in Sections \ref{sec_matrix_bingham} and \ref{sec_matrix_langevin} for the remainder terms in the corresponding asymptotic expansions appear to be the first such explicit bounds available in the literature. Therefore, we remark on the possible ramifications of those bounds for a variety of applications derived from a survey of the literature on applications of the matrix Bingham and matrix Langevin distributions.

The literature on such applications contains a broad range of reported values of $d$ and $p$. The case in which $(d,p) = (3,2)$ dates to analyses of data, drawn from astronomy and vectorcardiography, as provided in the classic article of Jupp and Mardia \cite{JuppMardia}. The same value, $(d,p) = (3,2)$ still appears in the recent literature, e.g., in data representing the orientation of two-dimensional elliptical paths of objects in near-Earth orbit (Jauch, Hoff and Dunson~\cite{JauchHoffDunson}) and various graph data sets (Mantoux et al.~\cite{Mantoux_etal}).

For larger values of $d$ and $p$, a variety of applications have appeared in the literature. Hoff~\cite{Hoff} applied the matrix Bingham distributions with $(d,p) = (270,2)$ to analyze data on protein-protein interaction networks.

Mantoux et al.~\cite{Mantoux_etal} provided illustrative examples with low- and high-dimensional synthetic data for $(d,p) = (3,2)$, $(20,5)$, $(20,10)$, and $(40,20)$; and with $d = 21$ and $p \in \{2,5,10\}$, they also analyzed resting-state functional magnetic resonance imaging (rs-fMRI) data from the UK Biobank repository of large-scale brain-imaging data set on tens of thousands of subjects (Sudlow et al.~\cite{Sudlow_etal}).

Holbrook, Vandenberg-Rodes and Shahbaba \cite{Holbrook_etal} considered Stiefel manifold data with values of $d$ ranging from $18$ to $53$. Their work included a study of data, collected during a memory experiment on rodents, in which neural spike activity was recorded in the hippocampi of six rats. Especially interesting is the analysis in \cite{Holbrook_etal} of the $53$-dimensional data on Super~Chris, a~particular rat whose neural spike data exhibited statistical features substantially distinct from the neural spike data of his fellow rats; cf.\ Granados-Garcia et al.~\cite{GranadosGarcia}.

Jauch et al.~\cite{JauchHoffDunson} analyzed with $(d,p) = (365,3)$ a meteorological data that included average daily temperatures for a collection of Canadian weather stations; for the case in which $p \ge 2$, this value of $d$ was the largest that we encountered in the literature.

In summary, for $p \ge 2$, we found applications in the literature with values of $(d,p)$ where $d$ ranges between $3$ and $365$, and $p$ ranges between $2$ and $20$. Further, for $p=1$, values of $d$ as large as $62,501$ were noted in \cite[Section~4]{BagyanRichards}.

The computation of the bounds in Theorems \ref{thm_Phi_m_bound_Bingham} and \ref{thm_Psi_m_bound_Langevin} is straightforward as it reduces to calculating the function $R_m(t)$, or the upper bound on $R_m(t)$, in \eqref{eq_R_m} and \eqref{eq_R_m_bound2}, respectively. As we noted earlier, these bounds are (strictly) decreasing in $d$ (and in $m$ if $t \le 1$), increasing in~$\gamma_0$ and $r$, and increasing in $p$. We reiterate that, in such calculations, the relationship between~$t$ and $d$ is that
\[
t = \begin{cases}
\gamma_0 \gamma_1 p^{1/2} (\tr A_+) d^{-(1-r)/2} & (\hbox{in Theorem \ref{thm_Phi_m_bound_Bingham}}), \\
2 \gamma_0 \gamma_1 p^{5/2} d^{-(3-r)/2} & (\hbox{in Theorem \ref{thm_Psi_m_bound_Langevin}}),
\end{cases}
\]
equivalently,
\[
d = \begin{cases}
\bigl(\gamma_0 \gamma_1 p^{1/2} (\tr A_+)/t\bigr)^{2/(1-r)} & (\hbox{in Theorem \ref{thm_Phi_m_bound_Bingham}}),\\
\bigl(2 \gamma_0 \gamma_1 p^{5/2}/t\bigr)^{2/(3-r)} & (\hbox{in Theorem \ref{thm_Psi_m_bound_Langevin}}).
\end{cases}
\]
For small values of $d$, e.g., $d=3$, the resulting values of $t$ usually are large and the bounds also can be large, indicating that larger choices for $m$ may be needed. For large values of $d$, e.g., $d=365$, the resulting values of $t$ typically are small, and the resulting bounds in terms of $R_m(t)$ quickly become negligible even for small values of $m$.

In summary, we expect that the results obtained here will find broad usage as they enable for the first time theoretical analyses and numerical calculations of bounds on the accuracy of the resulting expansions. In particular, these results enable calculation of the given expansions to any desired practical degree of accuracy.

\section{Proofs}\label{sec_proofs}

\subsection{Properties of the zonal polynomials}\label{sec_zonal_polys}

James \cite{James73} obtained for the zonal polynomials $C_\kappa(\Sigma)$ a crucial integral formula, as we now explain. Let \smash{$\Or(d) = \bigl\{H \in \R^{d \times d}\mid H'H = I_d\bigr\}$} be the group of orthogonal matrices in \smash{$\R^{d \times d}$}. For~${H \in \Or(d)}$, denote by $\dd H$ the Haar measure on $\Or(d)$, normalized to have total volume $1$. Also let ${\det}_j(\Sigma)$ be the $j$th principal minor of $\Sigma$, $j=1,\ldots,d$. For each partition $\kappa = (\kappa_1,\ldots,\kappa_d)$,
\begin{gather}
\label{eq_zonal}
C_\kappa(\Sigma) = C_\kappa(I_d) \int_{\Or(d)} \prod_{j=1}^{\ell(\kappa)} \bigl[{\det}_j\bigl(H' \Sigma H\bigr)\bigr]^{\kappa_j-\kappa_{j+1}} \dd H,
\end{gather}
where $\kappa_{d+1} \equiv 0$ and
\begin{equation}
\label{eq_zonal_identity}
C_\kappa(I_d) = 2^{2 |\kappa|} |\kappa|! (d/2)_\kappa \frac{\prod\limits_{1 \le i < j \le \ell(\kappa)} (2\kappa_i - 2\kappa_j - i + j)}{\prod\limits_{i=1}^{\ell(\kappa)} (2\kappa_i + \ell(\kappa) - i)!}.
\end{equation}
The formula \eqref{eq_zonal} was first given by James \cite[equation~(10.4)]{James73}, further applications of it were given in \cite{Kushner} and \cite{Richards85}, and the formula was extended to a more general setting in \cite[Theorem~4.8]{GrossRichards}.

By \eqref{eq_zonal}, $C_\kappa(\Sigma)$ is homogeneous of degree $|\kappa|$; moreover $C_\kappa(\Sigma)$ depends only on the $\ell(\kappa)$ non-zero parts of $\kappa$, so that $C_\kappa(\Sigma)$ is unchanged if zeros are appended to $\kappa$. Since any symmetric matrix can be diagonalized by an orthogonal transformation, then it also follows from \eqref{eq_zonal} that~$C_\kappa(\Sigma)$ is a symmetric function of the eigenvalues of $\Sigma$.

The importance of the normalization \eqref{eq_zonal_identity} is that the zonal polynomials satisfy the identity,
\begin{equation}
\label{eq_trace_zonal}
(\tr \Sigma)^k = \sum_{|\kappa|=k} C_\kappa(\Sigma),
\end{equation}
for all $k=0,1,2,\ldots$; see \cite[p.~479, equation~(21)]{James64}, \cite[p.~228, equation~(iii)]{Muirhead}, \cite[equation~(35.4.6)]{Richards10}. Also noteworthy is the special case of \eqref{eq_zonal_identity} in which $\kappa = (k)$, a partition with one part
\begin{equation}
\label{eq_zonal_identity_1part}
C_{(k)}(I_d) = \frac{(d/2)_k}{(1/2)_k}.
\end{equation}

The (seemingly formidable) expression for $C_\kappa(I_d)$ in \eqref{eq_zonal_identity} arises from the representation theory of $\mathfrak{S}_k$, the group of permutations on $k$ symbols, and we comment on the connection as follows. It is a classical result that the collection of irreducible representations of $\mathfrak{S}_k$ is in one-to-one correspondence with the set of partitions of $k$ (see Ledermann \cite[Chapter 4]{Ledermann}). For a~given partition $\kappa = (\kappa_1,\ldots,\kappa_d)$ we use the standard notation $\chi_{[2\kappa]}(1)$ for the dimension of the irreducible representation, corresponding to the partition $2\kappa = (2\kappa_1,\ldots,2\kappa_d)$, of the symmetric group $\mathfrak{S}_{2|\kappa|}$ on $2|\kappa|$ symbols. James \cite[pp.~478--479, equations~(18) and~(20)]{James64} proved that
\begin{equation}
\label{eq_zonal_identity_rep}
C_\kappa(I_d) = \frac{2^{2|\kappa|} |\kappa|! (d/2)_\kappa}{(2|\kappa|)!} \chi_{[2\kappa]}(1) = \frac{(d/2)_\kappa}{(1/2)_{|\kappa|}} \chi_{[2\kappa]}(1),
\end{equation}
and then \eqref{eq_zonal_identity} is derived from \eqref{eq_zonal_identity_rep} by applying Frobenius' famous formula for $\chi_{[2\kappa]}(1)$ (Ledermann \cite[p.~123]{Ledermann}).

Another consequence of \eqref{eq_zonal_identity_rep} is that
\begin{equation}
\label{eq_chi_kappa}
\frac{C_\kappa(I_d)}{(d/2)_\kappa} = \frac{\chi_{[2\kappa]}(1)}{(1/2)_{|\kappa|}}.
\end{equation}
Since the dimension of an irreducible representation does not change if zeros are appended to the corresponding partition, then it follows that, for any partition $\kappa$ such that $\ell(\kappa) \le p$, the right-hand side of \eqref{eq_chi_kappa} remains unchanged if $d$ is replaced by $p$; hence
\begin{equation}
\label{eq_kappa_invar}
\frac{C_\kappa(I_d)}{(d/2)_\kappa} = \frac{C_\kappa(I_p)}{(p/2)_\kappa}.
\end{equation}
We remark that the identity \eqref{eq_kappa_invar} can be verified directly, albeit laboriously, from \eqref{eq_zonal_identity}. Further, \eqref{eq_kappa_invar} has been noted explicitly in the literature, e.g., by Chikuse \cite[p.~30]{Chikuse} and by Edelman et al. \cite[p.~259]{Edelman}. Also, if $\ell(\kappa) = 1$, then \eqref{eq_kappa_invar} reduces to \eqref{eq_zonal_identity_1part}.

The zonal polynomials also have a combinatorial formulation. Recall that the set of all partitions is endowed with the \textit{lexicographic ordering}: For partitions $\kappa = (\kappa_1,\ldots,\kappa_d)$ and ${\lambda = (\lambda_1,\ldots,\lambda_d)}$ such that $|\kappa| = |\lambda|$, we say that $\lambda$ \textit{is less than} $\kappa$, written $\lambda < \kappa$, if there exists $i \in \{1,\ldots,d-1\}$ such that $(\lambda_1,\ldots,\lambda_{i-1}) = (\kappa_1,\ldots,\kappa_{i-1})$ and $\lambda_i < \kappa_i$.

For \smash{$\Sigma \in \R^{d \times d}_\sym$}, denote by $\sigma_1,\ldots,\sigma_d$ the eigenvalues of $\Sigma$. For each partition $\kappa = (\kappa_1,\ldots,\kappa_\ell)$ of length $\ell$, the {\it monomial symmetric function of $\Sigma$ corresponding to $\kappa$} is
\begin{equation}
\label{eq_msf}
M_\kappa(\Sigma) = \sum_{(j_1,\ldots,j_\ell)} \sigma_{j_1}^{\kappa_1} \cdots \sigma_{j_\ell}^{\kappa_\ell},
\end{equation}
where the sum is taken over all distinct permutations $(j_1,\ldots,j_\ell)$ of $\ell$ integers drawn from the set $\{1,\ldots,d\}$. James \cite{James68} (see also Muirhead \cite[Section 7.2]{Muirhead}) proved that
\begin{equation}
\label{eq_zonal_msf}
C_\kappa(\Sigma) = \sum_{\lambda \le \kappa} c_{\kappa,\lambda} M_\lambda(\Sigma),
\end{equation}
where the constants $c_{\kappa,\lambda}$ are \textit{nonnegative} and the summation is over all partitions $\lambda$ such that~${\lambda \le \kappa}$, i.e., $\lambda$ is less than or equal to $\kappa$ in the lexicographic ordering.

\subsection[Inequalities for the partitional shifted factorials and the zonal polynomials]{Inequalities for the partitional shifted factorials\\ and the zonal polynomials}
\label{sec_ineqs}

\begin{Lemma}
\label{lemma_factorial_ineq}
Let $a_1,\ldots,a_p$ be nonnegative integers such that $a_1+\cdots+a_p\ge 2$, and define
\begin{equation}
\label{eq_alpha_p}
\alpha_p = (2\pi)^{-(p-1)/4p} p^{1/4p}.
\end{equation}
Then
\begin{equation}
\label{eq_prod_factorial_bound}
\prod_{i=1}^p (p a_i)! \le \alpha_p^{2p} p^{p(a_1+\cdots+a_p)} \bigl[(a_1+\cdots+a_p)!\bigr]^p.
\end{equation}
\end{Lemma}

\begin{proof}
Since the multinomial coefficient, $(a_1+\cdots+a_p)/a_1! \cdots a_p!$, is a positive integer, then
\begin{equation}
\label{eq_multinomial}
\prod_{i=1}^p a_i! \le (a_1 + \cdots + a_p)!.
\end{equation}
It is well known that the classical gamma function is increasing on the interval $[2,\infty)$; therefore,
\smash{$\Gamma\bigl(a+(i-1)p^{-1}\bigr) \le \Gamma\bigl(a+(p-1)p^{-1}\bigr) = \Gamma\bigl(a + 1 - p^{-1}\bigr)$}
for $a \ge 2$ and all $i=1,\ldots,p$. Applying Gauss' multiplication formula for the gamma function, we obtain
\begin{align}
\Gamma(pa) &= (2\pi)^{-(p-1)/2} p^{pa-(1/2)} \prod_{i=1}^p \Gamma\bigl(a + (i-1)p^{-1}\bigr) \nonumber \\
&\le (2\pi)^{-(p-1)/2} p^{pa-(1/2)} \bigl[\Gamma\bigl(a + 1 - p^{-1}\bigr)\bigr]^p,\label{eq_Gauss_mult_1}
\end{align}
where the latter inequality holds for $a \ge 2$.

Setting $a = a_1+\cdots+a_p + p^{-1}$ in \eqref{eq_Gauss_mult_1}, we obtain
\begin{align}
\label{eq_Gauss_mult_2}
(p (a_1+\cdots+a_p))! &= \Gamma(p (a_1+\cdots+a_p) + 1) \nonumber \\
&\le (2\pi)^{-(p-1)/2} p^{p (a_1+\cdots+a_p) + (1/2)} [\Gamma(a_1+\cdots+a_p+1)]^p \nonumber \\
&= (2\pi)^{-(p-1)/2} p^{p (a_1+\cdots+a_p) + (1/2)} [(a_1+\cdots+a_p)!]^p.
\end{align}
Applying \eqref{eq_Gauss_mult_2} to \eqref{eq_multinomial}, we find that
\begin{align*}
\prod_{i=1}^p (p a_i)! &\le \bigl(p (a_1+\cdots+a_p)\bigr)! \le (2\pi)^{-(p-1)/2} p^{p (a_1+\cdots+a_p) + (1/2)} \bigl((a_1+\cdots+a_p)!\bigr)^p \\
&\equiv \alpha_p^{2p} p^{p(a_1+\cdots+a_p)} [(a_1+\cdots+a_p)!]^p.
\end{align*}
This completes the proof.
\end{proof}

\begin{Lemma}
\label{lemma_ineqs}
Let $d \ge p$ and let $\kappa = (\kappa_1,\ldots,\kappa_p)$ be a partition of length $\ell(\kappa) \le p$ and weight~${|\kappa| \ge 2}$. Then
\begin{equation}
\label{eq_d_kappa_ineq}
(2p)^{-|\kappa|} d^{|\kappa|} \le (d/2)_\kappa \le 2^{-|\kappa|} (d + |\kappa|)^{|\kappa|}.
\end{equation}
Also, with $\gamma_1 = \bigl(\sqrt{3}+1\bigr)/2 \simeq 1.366025$, we have
\begin{equation}
\label{eq_d_kappa_ratio_ineq3}
\prod_{i=1}^p \frac{\bigl(d^{1/2}/2\bigr)_{p \kappa_i}}{(d/2)_{p \kappa_i}} \le \bigl[\alpha_p (|\kappa|!)^{1/2} \gamma_1^{|\kappa|} p^{|\kappa|/2} d^{-|\kappa|/2}\bigr]^p.
\end{equation}
\end{Lemma}

\begin{proof}
By \eqref{eq_partit_fact}, we have
\begin{equation}
\label{eq_d_kappa_ineq_2}
(d/2)_\kappa = \prod_{j=1}^p \left(\frac12 d - \frac12(j-1)\right)_{\kappa_j}
= \prod_{j=1}^p \prod_{i=1}^{\kappa_j} \left(\frac12d-\frac12(j-1)+i-1\right).
\end{equation}
For all $i \ge 1$, $j \le p$, and $d \ge p$, we also have
\begin{align*}
\frac12d-\frac12(j-1)+i-1 &\ge \frac12d-\frac12(p-1) = \frac12 \left(1-\frac{p-1}{d}\right) d
\ge \frac12 \left(1-\frac{p-1}{p}\right) d = (2p)^{-1}d.
\end{align*}
Therefore, by \eqref{eq_d_kappa_ineq_2},
\[
(d/2)_\kappa \ge \prod_{j=1}^p \prod_{i=1}^{\kappa_j} \bigl[(2p)^{-1} d\bigr] = (2p)^{|\kappa|} d^{-|\kappa|}.
\]
This establishes the lower bound in \eqref{eq_d_kappa_ineq}.

Next, we apply to the product in \eqref{eq_d_kappa_ineq_2} the arithmetic-geometric mean inequality. Then we obtain
\begin{align*}
(d/2)_\kappa &\le \Bigg(\frac{1}{|\kappa|} \sum_{j=1}^p \sum_{i=1}^{\kappa_j} \left(\frac12d-\frac12(j-1)+i-1\right)\Bigg)^{|\kappa|} \\
&= \Bigg(\frac{1}{2 |\kappa|} \Bigg(d |\kappa| - \sum_{j=1}^p j \kappa_j + \sum_{j=1}^p \kappa_j^2\Bigg)\Bigg)^{|\kappa|} \le \Bigg(\frac{1}{2 |\kappa|} \Bigg(d |\kappa| + \Bigg(\sum_{j=1}^p \kappa_j\Bigg)^2\Bigg)\Bigg)^{|\kappa|} \\
&= 2^{-|\kappa|} (d + |\kappa|)^{|\kappa|}.
\end{align*}
This proves the upper bound in \eqref{eq_d_kappa_ineq}.

Further, it is known from \cite[Lemma A.2]{BagyanRichards} that, for all $r \ge 1$,
\smash{$\frac{(d^{1/2}/2)_r}{(d/2)_r} \le [(r-1)!]^{1/2} \gamma_1^{r-1} d^{-r/2}$},
and since $\gamma_1 > 1$, then we also have
\[
\frac{\bigl(d^{1/2}/2\bigr)_r}{(d/2)_r} \le [(r-1)!]^{1/2} \gamma_1^{r-1} d^{-r/2} \cdot \gamma_1 r^{1/2} = (r!)^{1/2} \gamma_1^r d^{-r/2}.
\]
Since the latter inequality is valid for $r = 0$, then we have shown that, for all $r \ge 0$,
\begin{equation}
\label{eq_d_kappa_ratio_ineq2}
\frac{\bigl(d^{1/2}/2\bigr)_r}{(d/2)_r} \le (r!)^{1/2} \gamma_1^r d^{-r/2}.
\end{equation}

Now we apply \eqref{eq_d_kappa_ratio_ineq2} to obtain
\begin{equation}
\label{eq_d_kappa_ratio_ineq}
\prod_{i=1}^p \frac{\bigl(d^{1/2}/2\bigr)_{p \kappa_i}}{(d/2)_{p \kappa_i}} \le \prod_{i=1}^p [(p \kappa_i)!]^{1/2} \gamma_1^{p \kappa_i} d^{-p \kappa_i/2}
= \gamma_1^{p |\kappa|} d^{-p |\kappa|/2} \prod_{i=1}^p [(p \kappa_i)!]^{1/2}.
\end{equation}
By \eqref{eq_prod_factorial_bound} in Lemma \ref{lemma_factorial_ineq} with $a_i = \kappa_i$, $i=1,\ldots,p$, we also have
\[
\prod_{i=1}^p (p \kappa_i)! \le \alpha_p^{2p} p^{p |\kappa|} (|\kappa|!)^p,
\]
and by substituting this bound into \eqref{eq_d_kappa_ratio_ineq}, we obtain \eqref{eq_d_kappa_ratio_ineq3}.
\end{proof}

For \smash{$\Sigma \in \R^{d \times d}_\sym$}, let $\sigma_1,\ldots,\sigma_d$ denote the eigenvalues of $\Sigma$ and define the diagonal matrix $\Sigma_+ = \diag(|\sigma_1|,\ldots,|\sigma_d|)$. The following result is a consequence of \eqref{eq_zonal_msf}, the combinatorial formula for the zonal polynomials.

\begin{Lemma}
\label{lem_zonal_psd}
Let $\kappa$ be a partition and \smash{$\Sigma \in \R^{d \times d}_\sym$}. Then $|C_\kappa(\Sigma)| \le C_\kappa(\Sigma_+)$.
\end{Lemma}

\begin{proof}
By \eqref{eq_zonal_msf}, the nonnegativity of the constants $c_{\kappa,\lambda}$, and the triangle inequality, we have
\[
|C_\kappa(\Sigma)| \le \boldsymbol{\bigg|}\sum_{\lambda \le \kappa} c_{\kappa,\lambda} M_\lambda(\Sigma)\boldsymbol{\bigg|}
\le \sum_{\lambda \le \kappa} c_{\kappa,\lambda} |M_\lambda(\Sigma)|.
\]
By \eqref{eq_msf}, for all partitions $\lambda$,
\[
|M_\lambda(\Sigma)| = \boldsymbol{\Big|}\sum \sigma_{j_1}^{\lambda_1} \cdots \sigma_{j_\ell}^{\lambda_\ell}\boldsymbol{\Big|}
\le \sum |\sigma_{j_1}|^{\lambda_1} \cdots |\sigma_{j_\ell}|^{\lambda_\ell}
= M_\lambda(\Sigma_+).
\]
Therefore, \[
|C_\kappa(\Sigma)| \le \sum_{\lambda \le \kappa} c_{\kappa,\lambda} M_\lambda(\Sigma_+) =
C_\kappa(\Sigma_+).
\]
The proof now is complete.
\end{proof}

The following result will be crucial for deriving bounds on the tails of the asymptotic series for the normalizing constants of the matrix Bingham and matrix Langevin distributions.

\begin{Proposition}
\label{prop_zonal_bound}
Let $\kappa$ be a partition such that $\ell(\kappa) \le p$, where $p \le d$, and let \smash{$\Sigma \in \R^{d \times d}_\sym$}. Then
\begin{gather}
\label{eq_zonal_bound}
|C_\kappa(\Sigma)| \le \Bigg(\prod_{i=1}^p \frac{\bigl(d^{1/2}/2\bigr)_{p \kappa_i}}{(d/2)_{p \kappa_i}}\Bigg)^{1/p} C_\kappa(I_d) \|\Sigma\|^{|\kappa|}.
\end{gather}
\end{Proposition}

\begin{proof}
By Lemma \ref{lem_zonal_psd}, $|C_\kappa(\Sigma)| \le C_\kappa(\Sigma_+)$. Since $\|\Sigma\| = \|\Sigma_+\|$, then it suffices to assume, without loss of generality, that $\Sigma = \Sigma_+$, i.e., that $\Sigma$ is positive semidefinite and diagonal, with diagonal entries $\sigma_1,\ldots,\sigma_d$.

Since $\Sigma$ is positive semidefinite, then so is every principal submatrix of $H'\Sigma H$, for all ${H \in \Or(d)}$. Therefore, ${\det}_j\bigl(H'\Sigma H\bigr) \ge 0$ for all $j=1,\ldots,d$.

For $\ell(\kappa) \le p$, it follows from \eqref{eq_zonal} that
\begin{equation}
\label{eq_zonal_ineq1}
C_\kappa(\Sigma) = C_\kappa(I_d) \int_{\Or(d)} \prod_{j=1}^p [{\det}_j\bigl(H'\Sigma H\bigr)]^{\kappa_j-\kappa_{j+1}} \dd H.
\end{equation}
Denote by $\bigl(H'\Sigma H\bigr)_{ii}$ the $i$th diagonal entry of $H'\Sigma H$, $i=1,\ldots,d$. By Hadamard's inequality for the principal minors of a positive semidefinite matrix \cite[p.~505]{Horn},
\[
{\det}_j\bigl(H'\Sigma H\bigr) \le \prod_{i=1}^j \bigl(H'\Sigma H\bigr)_{ii}
\]
for all $j=1,\ldots,d$; therefore,
\begin{align*}
\prod_{j=1}^p \bigl[{\det}_j\bigl(H'\Sigma H\bigr)\bigr]^{\kappa_j-\kappa_{j+1}} &\le \prod_{j=1}^p \prod_{i=1}^j \bigl[\bigl(H'\Sigma H\bigr)_{ii}\bigr]^{\kappa_j-\kappa_{j+1}} \\
&= \prod_{i=1}^p \prod_{j=i}^p \bigl[\bigl(H'\Sigma H\bigr)_{ii}\bigr]^{\kappa_j-\kappa_{j+1}}
= \prod_{i=1}^p \bigl[\bigl(H'\Sigma H\bigr)_{ii}\bigr]^{\kappa_i}.
\end{align*}
Inserting this bound into \eqref{eq_zonal_ineq1} and then applying H\"older's inequality, we obtain
\begin{align}
C_\kappa(\Sigma) &\le C_\kappa(I_d) \int_{\Or(d)} \prod_{i=1}^p \bigl[\bigl(H'\Sigma H\bigr)_{ii}\bigr]^{\kappa_i} \dd H \nonumber \\
&\le C_\kappa(I_d) \Bigg(\prod_{i=1}^p \int_{\Or(d)} \bigl[\bigl(H'\Sigma H\bigr)_{ii}\bigr]^{p \kappa_i} \dd H\Bigg)^{1/p}.\label{eq_zonal_Holder}
\end{align}

Denote by $h_{i,j}$ the $(i,j)$th entry of $H$, then $h_i = (h_{1,i},\ldots,h_{d,i})'$ is the $i$th column of $H$, $i=1,\ldots,d$. Since $\Sigma = \diag(\sigma_1,\ldots,\sigma_d)$, then for all $i=1,\ldots,d$,
\[
\bigl(H'\Sigma H\bigr)_{ii} = \sum_{l=1}^d h_{l,i}^2 \sigma_l = h_i'\Sigma h_i.
\]
Now regard $H$ as a random matrix having the uniform distribution (Haar measure) on $Or(d)$. Since the Haar measure is orthogonally invariant, and therefore invariant under permutation of the columns of $H$, it follows that the marginal distribution of $h_i$ does not depend on $i$. Denoting equality in distribution by \smash{$\eqdist$}, we obtain
\[
\bigl(H'\Sigma H\bigr)_{ii} = h_i'\Sigma h_i \eqdist h_1'\Sigma h_1 = \bigl(H'\Sigma H\bigr)_{11},
\]
for all $i=1,\ldots,d$. Consequently,
\begin{equation}
\label{eq_zonal_length1}
\int_{\Or(d)} \bigl[\bigl(H'\Sigma H\bigr)_{ii}\bigr]^{p \kappa_i} \dd H = \int_{\Or(d)} \bigl[\bigl(H'\Sigma H\bigr)_{11}\bigr]^{p \kappa_i} \dd H = \frac{C_{(p \kappa_i)}(\Sigma)}{C_{(p \kappa_i)}(I_d)},
\end{equation}
the second equality following from \eqref{eq_zonal}. Applying \eqref{eq_zonal_length1} to \eqref{eq_zonal_Holder}, we obtain
\begin{equation}
\label{eq_zonal_ineq2}
C_\kappa(\Sigma) \le C_\kappa(I_d) \Bigg(\prod_{i=1}^p \frac{C_{(p \kappa_i)}(\Sigma)}{C_{(p \kappa_i)}(I_d)}\Bigg)^{1/p}.
\end{equation}

By \cite[equation~(A.10)]{BagyanRichards},
\[
C_{(p \kappa_i)}(\Sigma) \le \frac{(d^{1/2}/2)_{p \kappa_i}}{(1/2)_{p \kappa_i}} \|\Sigma\|^{p \kappa_i};
\]
also, by \eqref{eq_zonal_identity_1part},
\[
C_{(p \kappa_i)}(I_d) = \frac{(d/2)_{p \kappa_i}}{(1/2)_{p \kappa_i}}.
\]
Substituting the two latter results into \eqref{eq_zonal_ineq2}, we obtain
\begin{align*}
C_\kappa(\Sigma) &\le C_\kappa(I_d) \Bigg(\prod_{i=1}^p \frac{\bigl(d^{1/2}/2\bigr)_{p \kappa_i}}{(1/2)_{p \kappa_i}} \|\Sigma\|^{p \kappa_i} \cdot \frac{(1/2)_{p \kappa_i}}{(d/2)_{p \kappa_i}}\Bigg)^{1/p} \\
&= C_\kappa(I_d) \|\Sigma\|^{|\kappa|} \Bigg(\prod_{i=1}^p \frac{\bigl(d^{1/2}/2\bigr)_{p \kappa_i}}{(d/2)_{p \kappa_i}}\Bigg)^{1/p},
\end{align*}
and this establishes \eqref{eq_zonal_bound}.
\end{proof}

\begin{Remark}
\label{rem_zonal_bound}
Suppose we evaluate both sides of the inequality \eqref{eq_zonal_bound} at $\Sigma = c I_d$, $c > 0$. Since~\smash{$C_\kappa(c I_d) = c^{|\kappa|} C_\kappa(I_d)$} and \smash{$\|c I_d\|^{|\kappa|} = c^{|\kappa|} d^{|\kappa|/2}$}, then the inequality reduces to
\begin{equation}
\label{eq_zonal_bound_BR}
1 \le \Bigg(\prod_{i=1}^p \frac{\bigl(d^{1/2}/2\bigr)_{p \kappa_i}}{(d/2)_{p \kappa_i}}\Bigg)^{1/p} d^{|\kappa|/2}.
\end{equation}
For $d \to \infty$, we have
\[
\frac{(d^{1/2}/2)_{p \kappa_i}}{(d/2)_{p \kappa_i}} \sim \frac{(d^{1/2}/2)^{p \kappa_i}}{(d/2)^{p \kappa_i}} = d^{-p \kappa_i/2},
\]
therefore
\[
\Bigg(\prod_{i=1}^p \frac{\bigl(d^{1/2}/2\bigr)_{p \kappa_i}}{(d/2)_{p \kappa_i}}\Bigg)^{1/p} d^{|\kappa|/2} \sim \Bigg(\prod_{i=1}^p d^{-p \kappa_i/2}\Bigg)^{1/p} d^{|\kappa|/2} = 1.
\]
Hence for scalar matrices $\Sigma$, both sides of the inequality \eqref{eq_zonal_bound} are asymptotically equal as ${d \to \infty}$. This indicates that, in high dimensions, the inequality is close to asymptotically tight if $\Sigma$ is within a small neighborhood of any scalar matrix.
\end{Remark}

\begin{Remark}
\label{rem_zonal_bound_2}
{\rm In the setting of general symmetric cones, another inequality for the zonal (spherical) polynomials was given by Faraut and Kor\'anyi \cite[Theorem XII.1.1]{FarautKoranyi} (see also the proof of~Lemma 6.5 in \cite{GrossRichards}, where the same inequality can be derived for the matrix cones). In the notation used here, the inequality in \cite[loc.~cit.]{FarautKoranyi} states that if $\kappa = (\kappa_1,\ldots,\kappa_p)$ is a partition of length $p$, $\Sigma$ is a $d \times d$ positive definite matrix, and $\sigma_{(1)} \ge \cdots \ge \sigma_{(d)}$ are the ordered eigenvalues of $\Sigma$, then
\begin{equation}
\label{eq_zonal_bound_FK}
C_\kappa(\Sigma) \le C_\kappa(I_d) \prod_{j=1}^p \sigma_{(j)}^{\kappa_j}.
\end{equation}

There is also the issue of which of the inequalities \eqref{eq_zonal_bound} and \eqref{eq_zonal_bound_FK} is sharper. We now provide two examples to show that the answer depends on $d$ and on $\Sigma$.

For the case in which $\Sigma = cI_d$ where $c$ does not depend on $d$, we find that \eqref{eq_zonal_bound_FK} is sharper than \eqref{eq_zonal_bound} since, by \eqref{eq_zonal_bound_BR},
\[
\prod_{i=1}^p \frac{\bigl(d^{1/2}/2\bigr)_{p \kappa_i}}{(d/2)_{p \kappa_i}} \ge d^{-p |\kappa|/2}.
\]
More generally, it can be seen that if $\Sigma$ is ``close'' to a scalar matrix and $d$ is small, then we can expect \eqref{eq_zonal_bound_FK} to be sharper than \eqref{eq_zonal_bound}, however both bounds are asymptotic to $d^{-p |\kappa|/2} C_\kappa(I_d)$ for large $d$.

On the other hand, suppose that $\Sigma$ is ``far'' from the scalar matrices $cI_d$ in the sense that $c$ does not depend on $d$ and $\|I_d\|/\|\Sigma\| \to \infty$ as $d \to \infty$. As an example, for constants $\sigma, \rho > 0$ and~${r \in [0,1)}$, let the eigenvalues of $\Sigma$ be
\[
\sigma_j = \begin{cases}
\sigma d^{r/2}, & j=1,\ldots,p, \\
\rho d^{-(d-p+1)/2}, & j=p+1,\ldots,d.
\end{cases}
\]
Then $\Sigma$ is non-singular and
\[
\|\Sigma\| = \Biggl(\sum_{j=1}^d \sigma_j^2\Biggr)^{1/2} = \bigl[p \sigma^2 d^{r} + \rho^2 (d-p) d^{-(d-p+1)}\bigr]^{1/2} \sim p^{1/2} \sigma d^{r/2},
\]
as $d \to \infty$, so $\Sigma$ satisfies the hypotheses of Theorem \ref{thm_Phi_m_bound_Bingham}. Further, $\Sigma$ is far from the scalar matrices $cI_d$ since \smash{$\|I_d\|/\|\Sigma\| \sim p^{-1/2} \sigma^{-1} d^{(1-r)/2} \to \infty$} as $d \to \infty$.

For this choice of $\Sigma$ and for a partition $\kappa$ of length $p$, the ratio of the right-hand sides of~\eqref{eq_zonal_bound} and \eqref{eq_zonal_bound_FK} equals
\[
\Bigg(\prod_{i=1}^p \frac{\bigl(d^{1/2}/2\bigr)_{p \kappa_i}}{(d/2)_{p \kappa_i}}\Bigg)^{1/p} \frac{\|\Sigma\|^{|\kappa|}}{\prod_{j=1}^p \sigma_{(j)}^{\kappa_j}} \sim \frac{\bigl(d^{1/2}/2\bigr)^{|\kappa|}}{(d/2)^{|\kappa|}} \frac{\bigl(p^{1/2} \sigma d^{r/2}\bigr)^{|\kappa|}}{\bigl(\sigma d^{r/2}\bigr)^{|\kappa|}} =
d^{-|\kappa|/2} p^{|\kappa|/2} \to 0
\]
as $d \to \infty$. In such instances, the bound in \eqref{eq_zonal_bound} is sharper than the bound in \eqref{eq_zonal_bound_FK} for sufficiently large $d$, and it becomes increasingly sharper afterwards as $d$ increases.
}\end{Remark}

Next, we obtain a lower bound for the zonal polynomials. This bound tightens an inequality stated in \cite[Theorem XII.1.1\,(ii)]{FarautKoranyi}. (Readers who wish to match the notation used in \cite{FarautKoranyi} with the notation in the present article will find it helpful to review \cite[Section 1.12]{DingGrossRichards}, where the notation used for general irreducible symmetric cones is made explicit in the case of each of the five classes of cones.)

\begin{Proposition}
\label{prop_zonal_lower_bound}
Let $\kappa = (\kappa_1,\ldots,\kappa_p)$ be a partition of length $p$, $\Sigma$ be a $d \times d$ positive definite matrix with ordered eigenvalues $\sigma_{(1)} \ge \cdots \ge \sigma_{(d)}$, and $N_d = \tfrac12(d+2)(d-1)$. Then
\begin{equation}
\label{eq_zonal_ineq_lb}
C_\kappa(\Sigma) \ge (1 + |\kappa|)^{-N_d} C_\kappa(I_d) \sigma_{(1)}^{\kappa_1} \cdots \sigma_{(p)}^{\kappa_p}.
\end{equation}
\end{Proposition}

\begin{proof}
By \cite[p.~241, line 4]{FarautKoranyi},
\begin{equation}
\label{eq_Phi_m_lb}
\frac{C_\kappa(\Sigma)}{C_\kappa(I_d)} \ge \frac{T_1}{T_2} \sigma_{(1)}^{\kappa_1} \cdots \sigma_{(p)}^{\kappa_p},
\end{equation}
where
\smash{$T_1 = \frac{1}{d_\kappa}$},
with $d_\kappa$ being the dimension of a certain vector space of polynomials that are homogeneous of degree $k = |\kappa|$, and
\begin{equation}
\label{eq_Delta_m_sq}
T_2 = \prod_{1 \le i < j \le p} \frac{B\bigl(\kappa_i - \kappa_j + \frac12 (j-i-1)+1,\frac12\bigr)}{B\bigl(\frac12 (j-i-1)+1,\frac12\bigr)}.
\end{equation}

In \cite[p.~242, line 9]{FarautKoranyi}, it is shown that
\[
d_\kappa \le \binom{N_d+k}{N_d} = \frac{(N_d+k)!}{N_d! k!} = \prod_{j=1}^{N_d} \frac{j+k}{j} \le (1+k)^{N_d}
\]
since $(j+k)/j \le 1+k$, for all $j=1,\ldots,N_d$. Therefore, \begin{equation}
\label{eq_Phi_m_sq}
T_1 = \frac{1}{d_\kappa} \ge (1 + |\kappa|)^{-N_d},
\end{equation}

As regards an upper bound for $T_2$ we now prove, using an elementary probabilistic approach, that
\begin{equation}
\label{eq_Delta_m_sq_ub}
T_2 \le 1.
\end{equation}
For $b, c > 0$, let $V$ denote a random variable having a beta distribution with probability density function
\[
\frac{v^{b-1} (1-v)^{c-1}}{B(b,c)} ,\qquad 0 < v < 1.
\]
 Then for $\alpha \ge 0$,
\[
\frac{B(\alpha+b,c)}{B(b,c)} = \frac{1}{B(b,c)} \int_0^1 v^{\alpha+b-1} (1-v)^{c-1} \dd v = \E\bigl(V^\alpha\bigr),
\]
and it is trivial that $\E\bigl(V^\alpha\bigr) \le 1$ since $\alpha \ge 0$ and $0 < V < 1$, almost surely. Setting $\alpha = \kappa_i - \kappa_j$, $b = \tfrac12 (j-i-1)+1$, and $c = \tfrac12$, we obtain
\begin{equation}
\label{eq_beta_distn_bound}
\frac{B\bigl(\kappa_i - \kappa_j+\frac12 (j-i-1)+1,\frac12\bigr)}{B\bigl(\frac12 (j-i-1)+1,\frac12\bigr)} \le 1
\end{equation}
for all $1 \le i < j \le p$. Applying \eqref{eq_beta_distn_bound} to each term in \eqref{eq_Delta_m_sq}, we obtain \eqref{eq_Delta_m_sq_ub}, and by applying~\eqref{eq_Phi_m_sq} and \eqref{eq_Delta_m_sq_ub} to \eqref{eq_Phi_m_lb}, then we obtain \eqref{eq_zonal_ineq_lb}.
\end{proof}

\subsection[The proofs for Section 2]{The proofs for Section \ref{sec_matrix_bingham}}
\label{sec_proofs_Phi_bounds}

\begin{proof}[Proof of Theorem \ref{thm_Phi_m_bound_Bingham}]
We begin by applying to \eqref{eq_remainder_term} the triangle inequality, the upper bound from Proposition \ref{prop_zonal_bound}, and the inequality $|C_\kappa(A)| \le C_\kappa(A_+)$. Then we obtain
\begin{align}
\label{eq_Phi_m_bound1}
|\Phi_{d,p;m}(A,\Sigma)| &\le \sum_{k=m}^\infty \frac{1}{k!} \sum_{|\kappa|=k} \frac{|C_\kappa(A)|}{C_\kappa(I_d)} \Bigg(\prod_{i=1}^p \frac{\bigl(d^{1/2}/2\bigr)_{p \kappa_i}}{(d/2)_{p \kappa_i}}\Bigg)^{1/p} C_\kappa(I_d) \|\Sigma\|^{|\kappa|} \nonumber \\
&\le \sum_{k=m}^\infty \frac{\|\Sigma\|^k}{k!} \sum_{|\kappa|=k} C_\kappa(A_+) \Bigg(\prod_{i=1}^p \frac{\bigl(d^{1/2}/2\bigr)_{p \kappa_i}}{(d/2)_{p \kappa_i}}\Bigg)^{1/p}.
\end{align}
By \eqref{eq_d_kappa_ratio_ineq3},
\[
\left(\prod_{i=1}^p \frac{\bigl(d^{1/2}/2\bigr)_{p \kappa_i}}{(d/2)_{p \kappa_i}}\right)^{1/p} \le \alpha_p (|\kappa|!)^{1/2} \gamma_1^{|\kappa|} p^{|\kappa|/2} d^{-|\kappa|/2},
\]
and by inserting this bound into \eqref{eq_Phi_m_bound1}, we obtain
\begin{align}
\label{eq_Phi_m_bound2}
\Phi_{d,p;m}(A,\Sigma)| &\le \sum_{k=m}^\infty \frac{\|\Sigma\|^k}{k!} \sum_{|\kappa|=k} C_\kappa(A_+) \alpha_p (|\kappa|!)^{1/2} \gamma_1^{|\kappa|} p^{|\kappa|/2} d^{-|\kappa|/2} \nonumber \\
&= \alpha_p \sum_{k=m}^\infty \frac{\gamma_1^k p^{k/2} d^{-k/2} \|\Sigma\|^k}{(k!)^{1/2}} \sum_{|\kappa|=k} C_\kappa(A_+).
\end{align}

By \eqref{eq_trace_zonal},
$
\sum_{|\kappa|=k} C_\kappa(A_+) = (\tr A_+)^k$,
and, by assumption, $\|\Sigma\| \le \gamma_0 d^{r/2}$. Therefore, we obtain from \eqref{eq_Phi_m_bound2} the inequality
\begin{equation}
\label{eq_Phi_m_bound3}
\Phi_{d,p;m}(A,\Sigma)| \le \alpha_p R_m\bigl(\gamma_0 \gamma_1 p^{1/2} (\tr A_+) d^{-(1-r)/2}\bigr),
\end{equation}
where
\begin{equation}
\label{eq_R_m}
R_m(t) = \sum_{k=m}^\infty \frac{t^k}{(k!)^{1/2}},\qquad t > 0.
\end{equation}
By applying the ratio test, we find that the series \eqref{eq_R_m} converges for all $t$, and therefore~\eqref{eq_Phi_m_bound3} converges for all $d$. This establishes the bound \eqref{eq_Phi_m_bound}.

To obtain a closed-form upper bound for $R_m(t)$ we begin by applying Stirling's well-known inequality for the factorial function \cite[Theorem 1]{Impens}; viz., for all $k \ge 1$,
\[
k! \ge (2\pi)^{1/2} k^{(2k+1)/2} {\rm e}^{-k}.
\]
Inverting Stirling's inequality and applying the result to \eqref{eq_R_m}, we obtain
\begin{equation}
\label{eq_R_m_bound}
R_m(t) \le (2\pi)^{-1/4} \sum_{k=m}^\infty k^{-(2k+1)/4} {\rm e}^{k/2} t^k.
\end{equation}
We now apply the telescoping method to obtain an upper bound for the right-hand side of~\eqref{eq_R_m_bound}. For $j \ge m$, let
\begin{equation}
\label{eq_aj}
a_j = j^{-(2j+1)/4} {\rm e}^{j/2} t^j,
\end{equation}
then we obtain through straightforward algebraic manipulations
\begin{equation}
\label{eq_aj_ratio}
\frac{a_{j+1}}{a_j} = \bigl(1+j^{-1}\bigr)^{-j/2} \frac{j^{1/4}}{(j+1)^{3/4}} {\rm e}^{1/2} t.
\end{equation}
Since the function \smash{$j \mapsto \bigl(1+j^{-1}\bigr)^{-j/2}$}, $j \ge m$, is decreasing, then
\begin{equation}
\label{eq_c_m}
\bigl(1+j^{-1}\bigr)^{-j/2} \le \bigl(1+m^{-1}\bigr)^{-m/2},
\end{equation}
for all $j \ge m$, and by applying the latter inequality to \eqref{eq_aj_ratio}, we obtain
\[
\frac{a_{j+1}}{a_j} \le \bigl(1+m^{-1}\bigr)^{-m/2} {\rm e}^{1/2} t \frac{j^{1/4}}{(j+1)^{3/4}}
= c_m^{1/2} t \frac{j^{1/4}}{(j+1)^{3/4}},
\]
where \smash{$c_m = \bigl(1+m^{-1}\bigr)^{-m} {\rm e}$}. It follows that, for all $k \ge m$,
\begin{align}
\label{eq_a_k}
a_k = a_m \prod_{j=m}^{k-1} \frac{a_{j+1}}{a_j}
\le a_m \prod_{j=m}^{k-1} c_m^{1/2} t \frac{j^{1/4}}{(j+1)^{3/4}}= a_m c_m^{(k-m)/2} t^{k-m} \prod_{j=m}^{k-1} \frac{j^{1/4}}{(j+1)^{3/4}}.
\end{align}
Since \smash{$\prod_{j=m}^{k-1} j = (k-1)!/(m-1)!$}, then
\begin{align}
\label{eq_a_k_factor}
\prod_{j=m}^{k-1} \frac{j^{1/4}}{(j+1)^{3/4}} &= \frac{[(k-1)!/(m-1)!]^{1/4}}{[k!/m!]^{3/4}}
= \frac{m^{3/4} [(m-1)!]^{1/2}}{k^{3/4} [(k-1)!]^{1/2}}.
\end{align}
On applying to \eqref{eq_a_k} the identity \eqref{eq_a_k_factor} and the explicit expression for $a_m$, as given in \eqref{eq_aj}, we obtain
\begin{align}
\label{eq_a_k_bound}
a_k &\le m^{-(2m+1)/4} {\rm e}^{m/2} t^m \cdot c_m^{(k-m)/2} t^{k-m} \frac{m^{3/4} [(m-1)!]^{1/2}}{k^{3/4} [(k-1)!]^{1/2}} \nonumber \\
&= ({\rm e}/m c_m)^{m/2} (m!)^{1/2} \frac{c_m^{k/2} t^k}{k^{3/4} [(k-1)!]^{1/2}}.
\end{align}
Therefore, by \eqref{eq_R_m_bound}, \eqref{eq_aj}, and \eqref{eq_a_k_bound},
\begin{align}
\label{eq_sums_2_bound}
R_m(t) &\le (2\pi)^{-1/4} ({\rm e}/m c_m)^{m/2} (m!)^{1/2} \sum_{k=m}^\infty \frac{c_m^{k/2} t^k}{k^{3/4} [(k-1)!]^{1/2}} \nonumber \\
&\le (2\pi)^{-1/4} ({\rm e}/m c_m)^{m/2} (m!)^{1/2} \Bigg(\sum_{k=m}^\infty \frac{1}{k^{3/2}}\Bigg)^{1/2} \Bigg(\sum_{k=m}^\infty \frac{c_m^k t^{2k}}{(k-1)!}\Bigg)^{1/2},
\end{align}
where the latter inequality follows by the Cauchy--Schwarz inequality.

To obtain an upper bound for the first sum in \eqref{eq_sums_2_bound}, we observe from graphical considerations using Riemann sums that, for $m \ge 2$,
\begin{equation}
\label{eq_zeta_32}
\sum_{k=m}^\infty \frac{1}{k^{3/2}} \le \int_m^\infty \frac{\dd t}{(t-1)^{3/2}} = \frac{2}{(m-1)^{1/2}} \le \left(\frac{8}{m}\right)^{1/2}.
\end{equation}
As for the second sum in \eqref{eq_sums_2_bound}, we apply the inequality
$
\frac{1}{(k+m-1)!} \le \frac{1}{k! (m-1)!}
$
to obtain
\begin{align}
\label{eq_partial_exp}
\sum_{k=m}^\infty \frac{t^{2k}}{(k-1)!} &= \sum_{k=0}^\infty \frac{t^{2(k+m)}}{(k+m-1)!} \le \frac{t^{2m}}{(m-1)!} \sum_{k=0}^\infty \frac{t^{2k}}{k!}
= \frac{t^{2m}}{(m-1)!} \exp\bigl(t^2\bigr).
\end{align}

Applying \eqref{eq_zeta_32} and \eqref{eq_partial_exp} to \eqref{eq_sums_2_bound}, we obtain
\begin{align}
R_m(t) &\le (2\pi)^{-1/4} ({\rm e}/m c_m)^{m/2} (m!)^{1/2} \left(\frac{8}{m}\right)^{1/4} \left(\frac{c_m^m t^{2m}}{(m-1)!} \exp\bigl(c_m t^2\bigr)\right)^{1/2} \nonumber \\
&= (4 {\rm e}/\pi)^{1/4} ({\rm e}/m)^{(m/2)-(1/4)} t^m \exp\bigl(c_m t^2/2\bigr),\label{eq_R_m_bound2}
\end{align}
and by inserting this bound at \eqref{eq_Phi_m_bound3}, we obtain \eqref{eq_Phi_m_boundw}.

Finally, it follows from \eqref{eq_Phi_m_bound} that \smash{$\Phi_{d,p;m}(A,\Sigma) = O\bigl(d^{-(1-r)m/2}\bigr)$} as $d \to \infty$.
\end{proof}

{\textbf{Comments on Remark \ref{rem_boundw_Bingham}}}.
Denote by $b_m$ the bound in \eqref{eq_R_m_bound2}. By direct calculation, we obtain
\begin{equation}
\label{eq_b_m_ratio}
\frac{b_{m+1}}{b_m} = t \Big(\frac{m}{m+1}\Big)^{(m/2)-(1/4)} \Big(\frac{\rm e}{m+1}\Big)^{1/2} \exp\bigl((c_{m+1}-c_m) t^2/2\bigr).
\end{equation}
By \eqref{eq_c_m}, $c_{m+1} - c_m < 0$; hence the latter three terms in \eqref{eq_b_m_ratio} each are strictly less than $1$, so we obtain $b_{m+1}/b_m < t$. Therefore, if $t \le 1$, then $b_{m+1} < b_m$, i.e., the sequence $\{b_m, m \ge 2\}$, is strictly decreasing; and by applying this result to \eqref{eq_Phi_m_boundw}, we obtain the statement in the first part of Remark~\ref{rem_boundw_Bingham}. We have also determined from simple numerical calculations that $b_m$ remains strictly decreasing for $t \le 1.5$.

For larger values of $t$, the sequence $b_m$ is unimodal. This can be proved by treating $m$ (temporarily) as a continuous variable, calculating the logarithmic derivative of \eqref{eq_b_m_ratio} with respect to $m$, and verifying that there exists an $m_0$ such that the derivative is positive for all~${m \le m_0}$ and negative for $m > m_0$. By graphing the bound \eqref{eq_R_m_bound2} as a function of $m$, we have observed that it appears to be unimodal for all $t \ge 1.6$.

\begin{proof}[Proof of Theorem \ref{thm_Phi_m_lower_bound_Bingham}]
Since $A$, $\Sigma$ are positive definite, then $C_\kappa(A) > 0$ and $C_\kappa(\Sigma) > 0$ for all partitions $\kappa$. By Proposition~\ref{prop_zonal_lower_bound},
\begin{align*}
\Phi_{d,p;m}(A,\Sigma) &= \sum_{k=m}^\infty \frac{1}{k!} \sum_{|\kappa|=k} \frac{C_\kappa(A) C_\kappa(\Sigma)}{C_\kappa(I_d)} \ge \sum_{k=m}^\infty \frac{1}{k!} \sum_{|\kappa|=k} C_\kappa(A) (1+|\kappa|)^{-N_d} \sigma_{(1)}^{\kappa_1} \cdots \sigma_{(p)}^{\kappa_p} \\
&= \sum_{k=m}^\infty \frac{(1+k)^{-N_d}}{k!} \sum_{|\kappa|=k} C_\kappa(A) \sigma_{(1)}^{\kappa_1} \cdots \sigma_{(p)}^{\kappa_p}.
\end{align*}
Since $\sigma_{(1)} \ge \cdots \ge \sigma_{(d)} > 0$, then \smash{$\sigma_{(1)}^{\kappa_1} \cdots \sigma_{(p)}^{\kappa_p} \ge \sigma_{(p)}^{|\kappa|}$} for any partition $\kappa$ of length $p$. Therefore, \[
\sum_{|\kappa|=k} C_\kappa(A) \sigma_{(1)}^{\kappa_1} \cdots \sigma_{(p)}^{\kappa_p} \ge \sigma_{(p)}^{k} \sum_{|\kappa|=k} C_\kappa(A) = \sigma_{(p)}^k (\tr A)^k,
\]
so we obtain
\[
\Phi_{d,p;m}(A,\Sigma) \ge \sum_{k=m}^\infty \frac{(1+k)^{-N_d} (\sigma_{(p)} \tr A)^k}{k!}.
\]

We now apply the telescoping method to bound the latter series. Let $\tau \equiv \sigma_{(p)} \tr(A)$ and, for~${j \ge m}$, define
\[
a_j = \frac{(1+j)^{-N_d} \tau^j}{j!}.
\]
Then
\[
\frac{a_{j+1}}{a_j} = \bigl(\frac{1+j}{2+j}\bigr)^{N_d} \frac{\tau}{1+j}.
\]
It is straightforward that, for $j \ge m$,
$\frac{1+j}{2+j} \ge \frac{1+m}{2+m}$,
and therefore
\[
\frac{a_{j+1}}{a_j} \ge \bigl(\frac{1+m}{2+m}\bigr)^{N_d} \frac{\tau}{1+j}.
\]
Hence for all $k \ge m$,
\begin{align*}
a_k = a_m \prod_{j=m}^{k-1} \frac{a_{j+1}}{a_j} &\ge a_m \prod_{j=m}^{k-1} \bigg[\left(\frac{1+m}{2+m}\right)^{N_d} \frac{\tau}{1+j}\bigg] = a_m \Big(\frac{1+m}{2+m}\Big)^{(k-m)N_d} \frac{m! \tau^{k-m}}{k!}.
\end{align*}
Substituting for $a_m$ and then simplifying, we obtain
\[
a_k \ge (2+m)^{mN_d} (1+m)^{-(1+m)N_d} \left(\frac{1+m}{2+m}\right)^{kN_d} \frac{\tau^k}{k!},
\]
and therefore
\[
\Phi_{d,p;m}(A,\Sigma) \ge \sum_{k=m}^\infty a_k \ge (2+m)^{mN_d} (1+m)^{-(1+m)N_d} \sum_{k=m}^\infty \left(\frac{1+m}{2+m}\right)^{kN_d} \frac{\tau^k}{k!}.
\]
The proof now is complete.
\end{proof}

\subsection[The proofs for Section 3]{The proofs for Section \ref{sec_matrix_langevin}}
\label{sec_proofs_Psi_bounds}

\begin{proof}[Proof of Theorem \ref{thm_Psi_m_bound_Langevin}]
Since $B'B$ is positive semidefinite, then all its eigenvalues are nonnegative; hence, by \eqref{eq_zonal_msf}, $C_\kappa\bigl(\frac14 B'B\bigr) \ge 0$. Also, it follows from \eqref{eq_partit_fact} that $(d/2)_\kappa > 0$. Therefore, it follows from \eqref{eq_Psi_m_Langevin} that $\Psi_{d,p;m}(B) \ge 0$.

By \cite[Section 5.6]{Horn}, $\|B'B\| \le \|B\|^2$. Denoting $\tfrac14 B'B$ by $\Lambda$, it follows from the assumption on~$B$ that
\smash{$\|\Lambda\| = \frac{1}{4} \|B'B\| \le \frac{1}{4} \|B\|^2 \le \gamma_0 d^{r/2}$}.
Since $\Lambda$ is positive semidefinite, then by Proposition~\ref{prop_zonal_bound},
\[
C_\kappa(\Lambda) \le \Bigg(\prod_{i=1}^p \frac{\bigl(d^{1/2}/2\bigr)_{p \kappa_i}}{(d/2)_{p \kappa_i}}\Bigg)^{1/p} C_\kappa(I_p) \|\Lambda\|^{|\kappa|}.
\]
Applying this bound to \eqref{eq_Psi_m_Langevin}, we find that
\[
\Psi_{d,p;m}(B) \le \sum_{k=m}^\infty \frac{\|\Lambda\|^k}{k!} \sum_{|\kappa|=k} \frac{C_\kappa(I_p)}{(d/2)_\kappa} \Bigg(\prod_{i=1}^p \frac{\bigl(d^{1/2}/2\bigr)_{p \kappa_i}}{(d/2)_{p \kappa_i}}\Bigg)^{1/p}.
\]

Next we apply the lower bound in \eqref{eq_d_kappa_ineq} and the bound in \eqref{eq_d_kappa_ratio_ineq3} to obtain
\begin{align*}
\Psi_{d,p;m}(B) &\le \sum_{k=m}^\infty \frac{\|\Lambda\|^k}{k!} \sum_{|\kappa|=k} C_\kappa(I_p) \cdot (2p)^{|\kappa|} d^{-|\kappa|} \cdot \alpha_p \gamma_1^{|\kappa|} d^{-|\kappa|/2} p^{|\kappa|/2} (|\kappa|!)^{1/2} \\
&= \alpha_p \sum_{k=m}^\infty \frac{\bigl(2 \gamma_1 p^{3/2} d^{-3/2} \|\Lambda\|\bigr)^k}{(k!)^{1/2}} \sum_{|\kappa|=k} C_\kappa(I_p).
\end{align*}
By \eqref{eq_trace_zonal} and the bound $\|\Lambda\| \le \gamma_0 d^{r/2}$, we deduce that
\begin{align*}
\Psi_{d,p;m}(B) &\le \alpha_p \sum_{k=m}^\infty \frac{\bigl(2 \gamma_1 p^{3/2} d^{-3/2} \cdot \gamma_0 d^{r/2}\bigr)^k}{(k!)^{1/2}} p^k = \alpha_p \sum_{k=m}^\infty \frac{\bigl(2 \gamma_0 \gamma_1 p^{5/2} d^{-(3-r)/2}\bigr)^k}{(k!)^{1/2}}.
\end{align*}
Noting that the latter series is identical with $R_m\bigl(2 \gamma_0 \gamma_1 p^{5/2} d^{-(3-r)/2}\bigr)$, where $R_m(t)$ is given in~\eqref{eq_R_m}, we apply the inequality \eqref{eq_R_m_bound2} to obtain \eqref{eq_Psi_m_boundw}.

Finally, it follows from \eqref{eq_Psi_m_bound} that $\Psi_{d,p;m}(B) = O\bigl(d^{-(3-r)m/2}\bigr)$ as $d \to \infty$.
\end{proof}

\begin{proof}[Proof of Theorem \ref{thm_Psi_m_lower_bound_Langevin}]
Since $B$ is of full rank, then $B'B$ is positive definite. Letting $\beta_{(1)} \ge \cdots \ge \beta_{(p)}$ denote the ordered eigenvalues of $\tfrac14 B'B$, then by Proposition \ref{prop_zonal_lower_bound},
\begin{align*}
\Psi_{d,p;m}(B) &= \sum_{k=m}^\infty \frac{1}{k!} \sum_{|\kappa|=k} \frac{C_\kappa\bigl(\frac14 B'B\bigr)}{(d/2)_\kappa} \ge \sum_{k=m}^\infty \frac{1}{k!} \sum_{|\kappa|=k} \frac{(1+|\kappa|)^{-N_p} C_\kappa(I_p)}{(d/2)_\kappa} \beta_{(1)}^{\kappa_1} \cdots \beta_{(p)}^{\kappa_p} \\
&\ge \sum_{k=m}^\infty \frac{(1+k)^{-N_p} \beta_{(p)}^k}{k!} \sum_{|\kappa|=k} \frac{C_\kappa(I_p)}{(d/2)_\kappa} \beta_{(1)}^{\kappa_1} \cdots \beta_{(p)}^{\kappa_p},
\end{align*}
since the ordering of the $\beta_{(j)}$ implies that \smash{$\beta_{(1)}^{\kappa_1} \cdots \beta_{(p)}^{\kappa_p} \ge \beta_{(p)}^{|\kappa|}$}.

By the upper bound in \eqref{eq_d_kappa_ineq}, we have
\[
\frac{1}{(d/2)_\kappa} \ge 2^{|\kappa|} (d+|\kappa|)^{-|\kappa|};
\]
therefore
\begin{align*}
\Psi_{d,p;m}(B) &\ge \sum_{k=m}^\infty \frac{(1+k)^{-N_p} \beta_{(p)}^k}{k!} \sum_{|\kappa|=k} 2^{|\kappa|} (d+|\kappa|)^{-|\kappa|} C_\kappa(I_p) \\
&= \sum_{k=m}^\infty \frac{(1+k)^{-N_p} 2^k (d+k)^{-k} \beta_{(p)}^k}{k!} \sum_{|\kappa|=k} C_\kappa(I_p) \\
&= \sum_{k=m}^\infty \frac{(1+k)^{-N_p} (d+k)^{-k}}{k!} (2p\beta_{(p)})^k.
\end{align*}
Since each term in the latter series is positive, then we obtain the single-term lower bound,
\[
\Psi_{d,p;m}(B) \ge \frac{(1+m)^{-N_p} (d+m)^{-m}}{m!} (2p\beta_{(p)})^m,
\]
and this bound also converges to $0$ as $d \to \infty$ or $m \to \infty$.

We apply now the telescoping method to obtain a sharper lower bound for $\Psi_{d,p;m}(B)$. Let~${\tau \equiv 2p\beta_{(p)}}$ and, for $j \ge m$, define
\smash{$a_j = (1+j)^{-N_p} (d+j)^{-j} \tau^j/j!$}.
Then
\begin{align*}
\frac{a_{j+1}}{a_j} &= \bigg(\frac{1+j}{2+j}\bigg)^{N_p} \bigg(\frac{d+j}{d+1+j}\bigg)^{j+1} \frac{\tau}{(j+1)(d+1+j)} \\
&\ge \bigg(\frac{1+m}{2+m}\bigg)^{N_p} {\rm e}^{-1} \frac{\tau}{(j+1)(d+1+j)}
\end{align*}
since
$
\frac{1+j}{2+j} \ge \frac{1+m}{2+m}$, \smash{$\bigl(\frac{d+j}{d+1+j}\bigr)^{j+1} \ge {\rm e}^{-1}$}
for all $j \ge m$ and all $d \ge 1$. Therefore, for all $k \ge m$,
\begin{align*}
a_k &= a_m \prod_{j=m}^{k-1} \frac{a_{j+1}}{a_j}
\ge a_m \prod_{j=m}^{k-1} \bigg[\bigg(\frac{1+m}{2+m}\bigg)^{N_p} \frac{{\rm e}^{-1} \tau}{(j+1)(d+1+j)}\bigg] \\
&= a_m \frac{\mu^{k-m}}{(m+1)_{k-m} (d+1+m)_{k-m}},
\end{align*}
where
\smash{$\mu = \bigl(\frac{1+m}{2+m}\bigr)^{N_p} {\rm e}^{-1} \tau$}.
Therefore, \begin{align*}
\Psi_{d,p;m}(B) \ge \sum_{k=m}^\infty a_k &\ge a_m \sum_{k=m}^\infty \frac{\mu^{k-m}}{(m+1)_{k-m} (d+1+m)_{k-m}} \\
&= a_m \sum_{k=0}^\infty \frac{\mu^k}{(m+1)_k (d+1+m)_k} \\
&= a_m {}_1F_2(1;m+1,d+1+m;\mu).
\end{align*}
This establishes \eqref{eq_Psi_m_lower_bound_Langevin}.
\end{proof}

\section{Concluding remarks}
\label{sec_conclusions}

In this article, we presented an approach, to calculating the normalizing constants of the matrix Bingham and matrix Fisher distributions, that is alternative to numerical simulation methods. Although Monte Carlo computations may be adequate for approximating the normalizing constants $\Phi_{d,p}$ and $\Psi_{d,p}$, such numerical simulations can only yield estimates of the errors in the computed values of the normalizing constants and cannot provide provable bounds on those errors.\looseness=-1

The approach given in the present article provides, under stated hypotheses, guaranteed upper bounds for the remainder terms, and therefore is able to calculate the normalizing constants to any desired level of accuracy. Further, the methods developed here provides explicit rates of convergence as $d \to \infty$ or $m \to \infty$, and also deduces in some instances lower bounds on the accuracy of each remainder term.

We note that the methods used in this article are applicable to a general class of distributions, on the Stiefel and Grassmann manifolds, as described by Chikuse \cite[pp.~35--36]{Chikuse}. Further, the approach given here applies to the complex matrix Bingham and complex matrix Langevin distributions; this can be done by making the necessary changes from the real zonal polynomials to the complex zonal polynomials and from the generalized hypergeometric functions of real symmetric matrix argument to the generalized hypergeometric functions of Hermitian matrix argument \cite{James64}. In particular, we note that the complex analog of \eqref{eq_zonal}, the crucial integral representation for the zonal polynomials, is provided in \cite[Theorem 4.8]{GrossRichards}.

Finally, we remark that all the results in this article extend, {\em mutatis mutandis}, to arbitrary symmetric cones. A basic principle observed from \cite{DingGrossRichards,GrossRichards} is that methods that are successful for the cone of real positive definite matrices will remain successful for any symmetric cone. Indeed, the results that we have established can be extended to derive rates of convergence of the spherical polynomial series expansions of all generalized hypergeometric functions defined on the symmetric cones.

\subsection*{Acknowledgements}

We are grateful to referees and an editor who read our manuscript with great thoroughness and provided us with incisive and perceptive comments.

\pdfbookmark[1]{References}{ref}
\LastPageEnding

\end{document}